\documentclass[12pt]{article}

\usepackage{latexsym,amsmath,amscd,amssymb,graphics,float}
\usepackage{enumerate}

\usepackage{graphicx,lscape}

\usepackage[colorlinks]{hyperref}
\usepackage{url}

\begin{document}

\newenvironment{proof}[1][Proof]{\textbf{#1.} }{\ \rule{0.5em}{0.5em}}

\newtheorem{theorem}{Theorem}[section]
\newtheorem{definition}[theorem]{Definition}
\newtheorem{lemma}[theorem]{Lemma}
\newtheorem{remark}[theorem]{Remark}
\newtheorem{proposition}[theorem]{Proposition}
\newtheorem{corollary}[theorem]{Corollary}
\newtheorem{example}[theorem]{Example}

\numberwithin{equation}{section}
\newcommand{\ep}{\varepsilon}
\newcommand{\R}{{\mathbb  R}}
\newcommand\C{{\mathbb  C}}
\newcommand\Q{{\mathbb Q}}
\newcommand\Z{{\mathbb Z}}
\newcommand{\N}{{\mathbb N}}

\newcommand{\bfi}{\bfseries\itshape}

\newsavebox{\savepar}
\newenvironment{boxit}{\begin{lrbox}{\savepar}
\begin{minipage}[b]{15.5cm}}{\end{minipage}\end{lrbox}
\fbox{\usebox{\savepar}}}

\title{{\bf On the coercivity of continuously differentiable vector fields}}
\author{R\u{a}zvan M. Tudoran}

\date{}
\maketitle \makeatother

\begin{abstract}
Given an arbitrary fixed continuously differentiable vector field on $\mathbb{R}^n$, we prove that this vector field is coercive if and only if its conservative part is coercive. We apply this result in order to provide sufficient conditions to guarantee the co-existence of equilibrium states of a continuously differentiable vector field and its conservative part.
\end{abstract}

\medskip



\section{Introduction}
\label{section:one}
The aim of this short note is to provide an alternative point of view regarding the classical notion of coercivity in the class of continuously differentiable vector fields on $\mathbb{R}^n$. 

More precisely, in the next section, using Presnov's decomposition of  continuously differentiable vector fields, we prove that an arbitrary given $\mathcal{C}^{1}$ vector field in $\mathbb{R}^n$ is coercive if and only if its conservative part is a coercive vector field. Consequently, we obtain that in the class of $\mathcal{C}^{1}$ vector fields in $\mathbb{R}^n$, the coercivity is determined actually \textbf{only} by the conservative part of the vector field. In the third section we present two consequences of the main result, each of them providing sufficient conditions to guarantee the co-existence of equilibrium states of a continuously differentiable vector field and its conservative part.

\section{Continuously differentiable coercive vector fields}

In this section we present the main result of this note which gives a characterization of the coercivity of a continuously differentiable vector field on $\mathbb{R}^n$ in terms of the coercivity of its conservative part. We start by recalling the definition of a \textbf{coercive} vector field on $\mathbb{R}^{n}$. 
\begin{definition}
A vector field $X\in\mathfrak{X}(\mathbb{R}^n)$ is called \textit{coercive} if 
$$
\lim_{\|\mathbf{x}\|\rightarrow \infty}\dfrac{\langle X(\mathbf{x}),\mathbf{x}\rangle}{\|\mathbf{x}\|}=\infty,
$$
where $\langle \cdot,\cdot \rangle$, $\| \cdot\|$, stand for the canonical inner product on $\mathbb{R}^{n}$, and the associated norm respectively.
\end{definition}

Let $X\in\mathfrak{X}(\mathbb{R}^n)$ be a continuously differentiable vector field. Recall from \cite{presnov} that $X$ can be decomposed uniquely in a conservative part and a sphere invariant part, such that
\begin{equation}\label{tdec}
X(\mathbf{x})=\nabla H_{X}(\mathbf{x})+u(\mathbf{x}), \forall\mathbf{x}\in\mathbb{R}^n,
\end{equation}
where $H_X(\mathbf{0})=0$ and $\langle u(\mathbf{x}), \mathbf{x}\rangle =0$, $\forall\mathbf{x}\in\mathbb{R}^n$. The potential function which generates the \textit{conservative} part of $X$, that is $\nabla H_X$, is given by
\begin{equation}\label{potential}
H_{X}(\mathbf{x})=\int_{0}^{1}\langle X(t\mathbf{x}),\mathbf{x}\rangle \mathrm{d}t, \forall\mathbf{x}\in\mathbb{R}^n.
\end{equation}
Note that this decomposition, known as Presnov's decomposition, is tight related to the canonical Euclidean structure of the ambient space. For a geometric extension of the Presnov decomposition \eqref{tdec}, see \cite{TDR}. 

A direct consequence of the decomposition \eqref{tdec} is the following characterization of coercivity of continuously differentiable vector fields on $\mathbb{R}^{n}$, which shows that \textit{the coercivity is determined actually only by the conservative part of the vector field}.

\begin{theorem}\label{mtm}
Let $X\in\mathfrak{X}(\mathbb{R}^n)$ be a continuously differentiable vector field on $\mathbb{R}^{n}$. Then $X$ is coercive if and only if its conservative part $\nabla H_{X}$ is coercive.
\end{theorem}
\begin{proof}
Using the decomposition \eqref{tdec} it follows that for any $\mathbf{x}\in\mathbb{R}^n$ 
$$
\langle X(\mathbf{x}),\mathbf{x}\rangle = \langle \nabla H_X (\mathbf{x})+u(\mathbf{x}),\mathbf{x}\rangle = \langle \nabla H_X (\mathbf{x}),\mathbf{x}\rangle + \langle u(\mathbf{x}),\mathbf{x}\rangle = \langle \nabla H_X (\mathbf{x}),\mathbf{x}\rangle,
$$
and hence
\begin{equation}\label{rel1}
\langle X(\mathbf{x}),\mathbf{x}\rangle = \langle \nabla H_X (\mathbf{x}),\mathbf{x}\rangle, \forall \mathbf{x}\in\mathbb{R}^n. 
\end{equation}
Consequently, we obtain that
\begin{equation*}
\lim_{\|\mathbf{x}\|\rightarrow \infty}\dfrac{\langle X(\mathbf{x}),\mathbf{x}\rangle}{\|\mathbf{x}\|}=\infty ~~~~\text{if and only if}~~~~ \lim_{\|\mathbf{x}\|\rightarrow \infty}\dfrac{\langle \nabla H_X(\mathbf{x}),\mathbf{x}\rangle}{\|\mathbf{x}\|}=\infty.
\end{equation*}
\end{proof}

Let us recall from \cite{presnov} that the uniqueness of the decomposition \eqref{tdec} is guaranteed by the fact that $X$ is continuously differentiable on the whole $\mathbb{R}^n$. Nevertheless the equality \eqref{rel1} holds true for an arbitrary continuously differentiable vector field defined on an open subset of $\mathbb{R}^{n}$, star--shaped with respect to the origin.

\begin{proposition}\label{prpok}
Let $X\in\mathfrak{X}(\Omega)$ be a continuously differentiable vector field defined on an open set $\Omega\subseteq\mathbb{R}^{n}$, star--shaped with respect to the origin. Then the following relation holds true
\begin{equation*}
\langle X(\mathbf{x}),\mathbf{x}\rangle = \langle \nabla H_X (\mathbf{x}),\mathbf{x}\rangle, \forall \mathbf{x}\in\Omega,
\end{equation*}
where 
\begin{equation*}
H_{X}(\mathbf{x})=\int_{0}^{1}\langle X(t\mathbf{x}),\mathbf{x}\rangle \mathrm{d}t, \forall\mathbf{x}\in\Omega.
\end{equation*}
\end{proposition}
\begin{proof}
Let $h_X \in\mathcal{C}^{1}(\Omega,\mathbb{R})$ be given by $h_{X}(\mathbf{x}):=\langle X(\mathbf{x}),\mathbf{x}\rangle$, $\forall\mathbf{x}\in\Omega$. Using this notation, the following relation holds true
$$
H_X (\mathbf{x})=\int_{0}^{1}\dfrac{1}{t}h_X(t\mathbf{x})\mathrm{d}t, \forall\mathbf{x}\in\Omega.
$$
Thus, for any $\mathbf{x}=(x_1,\dots,x_n)\in\Omega$ we obtain successively that
\begin{align*}
\langle\nabla H_{X}(\mathbf{x}),\mathbf{x}\rangle &= \sum_{i=1}^{n} x_i \dfrac{\partial H_X}{\partial x_i}(\mathbf{x})=\sum_{i=1}^{n} x_i 
\dfrac{\partial}{\partial x_i}\int_{0}^{1}\dfrac{1}{t}h_X (t\mathbf{x})\mathrm{d}t=\sum_{i=1}^{n} x_i \int_{0}^{1}\dfrac{\partial h_X}{\partial x_i} (t\mathbf{x})\mathrm{d}t\\
&=\int_{0}^{1}\left[ \sum_{i=1}^{n} x_i \dfrac{\partial h_X}{\partial x_i} (t\mathbf{x})\right]\mathrm{d}t=\int_{0}^{1}\left[\dfrac{\mathrm{d}}{\mathrm{d}t}h_X(t\mathbf{x})\right]\mathrm{d}t=h_X(\mathbf{x})=\langle X(\mathbf{x}),\mathbf{x}\rangle,
\end{align*}
and hence we get the conclusion.
\end{proof}

\section{Two results concerning the existence of equilibria of vector fields}

In this section we present two consequences of the results from the previous section, regarding the co-existence of equilibrium states of a continuously differentiable vector field and its conservative part. In order to do that we recall a classical result which provides a boundary condition that guarantees the existence of solutions to the equation $f(\mathbf{x})=\mathbf{0}$, where $f:\overline{B}_{r}(\mathbf{0})\rightarrow\mathbb{R}^{n}$ is a continuous function defined on the $n-$dimensional closed ball of radius $r>0$ centered at the origin, $\overline{B}_{r}(\mathbf{0}):=\{\mathbf{x}\in\mathbb{R}^{n}:~ \|\mathbf{x}\|\leq r\}$. For details regarding the proof of this result, see e.g. \cite{francu}.

\begin{theorem}[\cite{francu}]\label{TFR} Let $f:\overline{B}_{r}(\mathbf{0})\rightarrow\mathbb{R}^{n}$ be a continuous mapping satisfying the boundary condition 
\begin{equation*}
\langle f(\mathbf{x}),\mathbf{x}\rangle>0, ~ \forall\mathbf{x}\in \partial\overline{B}_{r}(\mathbf{0}).
\end{equation*}
Then the equation $f(\mathbf{x})=\mathbf{0}$ admits at least one solution $\mathbf{x}\in B_{r}(\mathbf{0})$.
\end{theorem}

Next, we prove that in the case of a continuously differentiable vector field $X$, the hypothesis of Theorem \ref{TFR} forces the conclusion to holds true not only for the vector field $X$, but also for its conservative part $\nabla H_X$, where the formula of the potential function $H_X$ is given by the equality \eqref{potential}. More precisely, we obtain the following result.
\begin{proposition} 
Let $X\in\mathfrak{X}(\overline{B}_{r}(\mathbf{0}))$ be a $\mathcal{C}^{1}$ vector field defined on the closed ball $\overline{B}_{r}(\mathbf{0})\subset\mathbb{R}^{n}$, satisfying the boundary condition 
\begin{equation}\label{bdcond}
\langle X(\mathbf{x}),\mathbf{x}\rangle>0, ~ \forall\mathbf{x}\in \partial\overline{B}_{r}(\mathbf{0}).
\end{equation}
Then both the vector field $X$ and its conservative part $\nabla H_X$, admits at least one equilibrium state in $B_{r}(\mathbf{0})$.
\end{proposition}
\begin{proof}
Using the Proposition \ref{prpok} we get that $\langle \nabla H_X (\mathbf{x}),\mathbf{x}\rangle = \langle X(\mathbf{x}),\mathbf{x}\rangle, ~ \forall\mathbf{x}\in\overline{B}_{r}(\mathbf{0}),$ and hence the boundary condition \eqref{bdcond} becomes
$$
\langle \nabla H_X (\mathbf{x}),\mathbf{x}\rangle = \langle X(\mathbf{x}),\mathbf{x}\rangle>0, ~ \forall\mathbf{x}\in \partial\overline{B}_{r}(\mathbf{0}).
$$
Now the conclusion follows by applying Theorem \ref{TFR} to the vector field $X$ and also to its conservative part $\nabla H_X$.
\end{proof}

Next result presents a sufficient condition that guarantees the existence of equilibrium states of a continuously differentiable vector filed and its conservative part, when both are perturbed by an arbitrary constant vector field. 
\begin{proposition} 
Let $X\in\mathfrak{X}(\mathbb{R}^{n})$ be a $\mathcal{C}^{1}$ coercive vector field. Then for any fixed $\mathbf{b}\in\mathbb{R}^{n}$, each of the vector fields $X+\mathbf{b}$ and $\nabla H_X +\mathbf{b}$, admits at least one equilibrium state.
\end{proposition}
\begin{proof}
Let us start by fixing an arbitrary vector $\mathbf{b}\in\mathbb{R}^{n}$. Using the Theorem \ref{mtm} we obtain that the conservative part of $X$ is also a coercive vector field, and hence using the definition of coercivity, there exists $\rho=\rho(\mathbf{b})>0$ such that 
$$
\langle \nabla H_X (\mathbf{x})+\mathbf{b},\mathbf{x}\rangle = \langle X(\mathbf{x})+\mathbf{b},\mathbf{x}\rangle>0, ~ \forall\mathbf{x}\in \partial\overline{B}_{\rho}(\mathbf{0}).
$$
Now the conclusion follows by Theorem \ref{TFR} applied to each of the vector fields $X+\mathbf{b}$ and $\nabla H_X + \mathbf{b}$.
\end{proof}

\bigskip
\bigskip

\noindent {\sc R.M. Tudoran}\\
West University of Timi\c soara\\
Faculty of Mathematics and Computer Science\\
Department of Mathematics\\
Blvd. Vasile P\^arvan, No. 4\\
300223 - Timi\c soara, Rom\^ania.\\
E-mail: {\sf razvan.tudoran@e-uvt.ro}\\
\medskip

\end{document}